\documentclass[a4paper,12pt]{amsart}

\allowdisplaybreaks

\usepackage{amsthm,amsmath,amssymb,amsxtra}
\usepackage{bm,mathrsfs,marvosym,eurosym}
\usepackage{mathtools}

\usepackage{enumerate,comment,graphicx,psfrag}
\usepackage{xspace}

\usepackage{color}

\usepackage{hhline}

\usepackage{hyperref}

\usepackage{nameref}
\usepackage{bm}

\usepackage{verbatim}
\usepackage{tabls}

\def\MR#1{\href{http://www.ams.org/mathscinet-getitem?mr=#1}{MR#1}}

\begin{document}
\theoremstyle{plain}% default
\newtheorem{theorem}{Theorem}[section]
\newtheorem{lemma}{Lemma}[section]
\newtheorem{proposition}{Proposition}[section]
\newtheorem{corollary}{Corollary}[section]

\theoremstyle{definition}
\newtheorem{definition}[theorem]{Definition}
\newtheorem{example}[theorem]{Example}

\newtheorem{remark}{Remark}[section]
\newtheorem{remarks}[remark]{Remarks}
\newtheorem{note}{Note}
\newtheorem{case}{Case}

\numberwithin{equation}{section}
\numberwithin{table}{section}
\numberwithin{figure}{section}

%\renewcommand\baselinestretch{1.1}

%%%%%%%%%%%%%% ***** definitions  *************

\def\div{\text{\rm div}}
\def\cal{\Cal } 
\def\d{{\mathrm d}}
\def\tr|{|\! |\! |}

\def\i{\mathrm{i} }
\def\e{\mathrm{e} }
\def\C{{\mathbb C}}
\def\P{{\mathbb P}}
\def\R {{\mathbb R}}
\def\N{{\mathbb N}}
\def\Z{{\mathbb Z}}
\def\N{{\mathbb N}}
\def\T{{\mathbb T}}
\def\Re{\text {\rm Re} }

\def\wU{\widehat U}
\def\ww{\widehat w}
\def\wR{\widehat R}
\def\he{\hat e}

\def\D{{\mathscr{D}}}

\def\A{{\mathcal{A}}}
\def\E{{\mathcal{E}}}
\def\L{{\mathcal{L}}}
\def\M{{\mathcal{M}}}
\def\X{{\mathcal{X}}}
\def\V{{\mathcal{V}}}
\def\U{{\mathcal{U}}}
\def\W{{\mathcal{W}}}

\def\<{{\langle }}
\def\>{{\rangle }}

\def\wU{\widehat U}
\def\wV{\widehat V}

\def\AA{\text{${\mathcal O}\hskip-3pt\iota$}}

\def\<{{\langle }}
\def\>{{\rangle }}

\newcommand{\stig}[1]{{\color{red}{#1}}}

%***********************************************************

\title[DG methods for parabolic equations]
{A priori and a posteriori error estimates\\
  for discontinuous Galerkin time-discrete methods\\
  via maximal regularity}

\author[Georgios Akrivis]{Georgios Akrivis}
\address{Department of Computer Science and Engineering, University of Ioannina, 
451$\,$10 Ioannina, Greece, and Institute of Applied and Computational Mathematics, 
FORTH, 700$\,$13 Heraklion, Crete, Greece}
\email {\href{mailto:akrivis@cse.uoi.gr}{akrivis{\it @\,}cse.uoi.gr}}

\author[Stig Larsson]{Stig Larsson}
\address{Department of Mathematical Sciences, Chalmers University of Technology
and University of G\"oteborg, G\"oteborg, Sweden}
\email {\href{mailto:stig@chalmers.se}{stig{\it @\,}chalmers.se}}  
\date{\today}



\thanks{S.~L.~was supported by Vetenskapsr\aa{}det (VR) through grant
  no.~2017-04274.}

\begin{abstract}
  The maximal regularity property of discontinuous Galerkin methods
  for linear parabolic equations is used together with variational
  techniques to establish a priori and a posteriori error estimates of
  optimal order under optimal regularity assumptions. The analysis is
  set in the maximal regularity framework of UMD Banach spaces.
  Similar results were proved in an earlier work, based on the
  consistency analysis of Radau~IIA methods. The present error
  analysis, which is based on variational techniques, is of
  independent interest, but the main motivation is that it extends to
  nonlinear parabolic equations; in contrast to the earlier work. Both
  autonomous and nonautonomous linear equations are considered.
\end{abstract}

%***************%%%%%%%%%%%%%%%%%%%%%%%%%%%%%%%%%%%%%
%%%%%%%%%%%%%%%%%%%%%%%%%%%%%%%%%%%%%%%%%%%%%%%%%%%%%

\maketitle

%%%%%%%%%%%%%%%%%%%%%%%%%%%%%%%%%%%%%%%%%%

\section{Introduction}\label{Se:1}

We consider the discretization of differential equations satisfying
the maximal parabolic $L^p$-regularity property in unconditional
martingale differences (UMD) Banach spaces by discontinuous Galerkin
(dG) methods.  We combine the maximal regularity property of the
methods with variational techniques and establish optimal order,
optimal regularity, a priori and a posteriori error estimates.

\subsection{Maximal parabolic regularity}\label{SSe:1.1}
We consider an initial value problem for a linear parabolic equation,
\begin{equation}
\label{ivp}
\left \{
\begin{aligned} 
&u' (t) +A u(t) =f(t), \quad 0<t< T,\\
&u(0)=u_0,
\end{aligned}
\right .
\end{equation}
in a UMD Banach space $X$ with initial value $u_0\in X.$ Our
structural assumption is that the closed operator $-A$ is the generator
of an analytic semigroup on $X$ having \emph{maximal
  $L^p$-regularity}. This means that for vanishing initial value
$u_0=0$, for any $T\in (0,\infty],$ for some, or, as it turns out, for
all $p\in (1,\infty),$ and for any $f\in L^p((0,T);X)$ there exists a
unique solution $u$ of \eqref{ivp} such that $u'\in L^p((0,T);X)$;
then we also have $Au\in L^p((0,T);X).$ As a consequence of the closed
graph theorem, the solution $u$ of \eqref{ivp} with $u_0=0$ satisfies
the stability estimate
\begin{equation}
\label{max-reg}
\|u'\|_{L^p((0,T);X)}+\|Au\|_{L^p((0,T);X)}\leqslant c_{p,X}\|f\|_{L^p((0,T);X)}
\quad\forall f\in L^p((0,T);X)
\end{equation}
with a constant $c_{p,X}$ independent  of $T,$ depending only on $p$ and $X$;
see, e.g., \cite{DHP} and \cite{KuW}. 

Since also $\|f\|_{L^p((0,T);X)}\leqslant \|u'\|_{L^p((0,T);X)}+\|Au\|_{L^p((0,T);X)}$
by the triangle inequality, the norm of the sum $ \|u'+Au\|_{L^p((0,T);X)}$
and the sum of the norms $\|u'\|_{L^p((0,T);X)}+\|Au\|_{L^p((0,T);X)}$ are equivalent 
on the Banach space
\begin{equation*}
  \big\{v\in W^{1,p}((0,T);X)\cap L^p((0,T);\D(A)): v(0)=0 \big\}
\end{equation*}
with constants independent of $T,$ for $T<\infty$.  Here $\D(A):=\{v\in X: Av \in X\}$
is the domain of the operator $A$.

\subsection{The numerical methods}\label{SSe:1.2}
We consider the discretization of the initial value problem \eqref{ivp} 
by dG methods. 

Let $N\in \N, k=T/N$ be the constant time step,
$t_n:=nk, n=0,\dotsc,N,$ be a uniform partition of the time interval
$[0,T],$ and $J_n:= (t_n,t_{n+1}].$

For $s\in \N_0,$ we denote by $\P(s)$ and $\P_{X'}(s)$ the spaces of
polynomials of degree at most $s$ with coefficients in $\D(A)$ and in
the dual $X'$ of $X$, respectively, i.e., the elements $g$ of $\P(s)$
and of $\P_{X'}(s)$, respectively, are of the form
\begin{equation*}
  g(t)= \sum_{j=0}^s  t^j w_j, \quad
  w_j\in  \D(A)\quad\text{and}\quad w_j\in  X', \quad j=0,\dotsc, s.
\end{equation*}
With this notation, let $\V_k^{\text{c}} (s)$ and
$\V_k^{\text{d}} (s)$ be the spaces of continuous and possibly
discontinuous piecewise elements of $\P(s)$, respectively,
\begin{align*}
&\V_k^{\text{c}} (s):=\{v\in C\big ([0,T];\D(A)\big ): v|_{J_n}\in \P(s), \ n=0,\dotsc, N-1\},\\
&\V_k^{\text{d}} (s):=\{v: (0,T]\to \D(A), \ v(0)\in X,\ v|_{J_n}\in \P(s), \ n=0,\dotsc, N-1\}.
\end{align*}
The spaces $\X_k^{\text{c}} (s)$ and $\X_k^{\text{d}} (s)$ are defined
analogously, with coefficients $w_j\in X$.  Let us emphasize that we
allow the values at $0$ of functions in $\V_k^{\text{d}} (s)$ to belong
to $X.$

We denote by $\langle\cdot,\cdot\rangle$ the duality pairing between $X$ and $X'.$

For $q\in \N,$ with starting value $U(0)=U_0:=u_0\in X$ and source term $f\in L^p((0,T);X)$,
we  consider the discretization of the initial value problem \eqref{ivp}  by the 
\emph{discontinuous Galerkin method} dG$(q-1)$, i.e., we seek $U\in \V_k^{\text{d}} (q-1)$ such that
\begin{equation}
\label{dg}
\int_{J_n}   \big( \< U' ,v \>  + \< AU ,v \> \big) \, \d t  
+ \< U_n^{+}-U_n, v_n^{+}\>
= \int_{J_n} \<f,v\>\, \d t \quad \forall v \in \P_{X'}(q-1)
\end{equation}
for $n=0,\dotsc,N-1$.  As usual, we use the notation $v_n:=v(t_n),$
$v_n^{+}:=\lim_{s\searrow 0} v(t_n+s)$.

\subsubsection{A reconstruction operator}\label{SSSe:1.2.1}
With $0<c_1<\dotsb<c_q=1$ the Radau nodes in the interval $[0,1]$, 
let $t_{ni}:=t_n+c_ik, i=1,\dotsc,q,$ be the intermediate nodes; 
we also use the notation $t_{n0}:=t_n.$
The \emph{reconstruction operator} $\V_k^{\text{d}} (q-1) \to \X_k^{\text{c}} (q), w\mapsto \ww,$ 
defined via extended interpolation at the Radau nodes, cf.\ \cite{MN},
\begin{equation}
\label{eq:ext-interp}
\ww ( t_{nj})= w( t_{nj}), \quad j=0, \dotsc ,q \quad (w(t_{n0})=w_n),
\end{equation}
plays a crucial role in our analysis; notice that
$\ww\in \V_k^{\text{c}} (q),$ provided $w(0)\in \D(A).$ 
%It is well known that
%%
%\begin{equation}
%\label{wU_def}
%\int_{J_n}    \< \ww' ,v \>\,\d t = \int_{J_n}    \< w' ,v \>   \, \d t
%+ \< w_n^+-w_n, v_n^+\>  \quad \forall v \in \P_{X'}(q-1) ;
%\end{equation}
%%
%this can be seen using the Lagrange form of the difference
%$\ww(t)-w(t)=(w_n-w_n^+)\ell_{n0}(t),$ with $ \ell_{n0}$ the
%polynomial of degree $q$ vanishing at the nodes $t_{n1},\dotsc,t_{nq}$
%and taking the value $1$ at the node $t_{n0},$ integration by parts,
%and the fact that the Radau quadrature rule with $q$ nodes integrates
%polynomials of degree at most $2q-2$ exactly.

\subsection{Main results}\label{SSe:1.3}
We establish the following optimal order, optimal regularity, a priori and
a posteriori error estimates.

\begin{theorem}[A priori error estimates]\label{Theorem1} 
  Let $p\in(1,\infty)$ and assume that the solution of
  \eqref{ivp} is sufficiently regular,
  $u\in W^{q,p}\big ((0,T);\D(A)\big ).$ Then, the dG approximation
  $U\in \V_k^{\text{d}} (q-1)$ satisfies the estimate
  \begin{equation}
    \label{apriori-estimate1}
    \|A(u-U)\|_{L^p((0,T);X)}  \leqslant C k^q\|Au^{(q)}\|_{L^p((0,T);X)}.
  \end{equation}
  Furthermore, if $u_0\in \D(A)$ and
  $u\in W^{q+1,p}\big ((0,T);X\big ),$ for the reconstruction
  $\wU\in \V_k^{\text{c}} (q),$ we have
  \begin{equation}
    \label{apriori-estimate2}
    \begin{aligned}
      & \|(u - \widehat U)'\|_{L^p((0,T);X)} + \|A(u-\wU)\|_{L^p((0,T);X)}
      \\ & \qquad
      \leqslant
      C k^q\big( \|u^{(q+1)}\|_{L^p((0,T);X)}+ \|Au^{(q)}\|_{L^p((0,T);X)}\big).
    \end{aligned}
  \end{equation}
  The constant $C$ depends on $q, p$, and $X$, but it is independent of
  the solution $u,$ of $T,$ and of the time step $k.$
\end{theorem}

\begin{theorem}[A posteriori error estimate]\label{Theorem2} 
  For initial value $u_0\in \D(A),$ let
  $R(t):= \widehat U'(t)+A\widehat U(t)-f(t)$ be the residual of the
  reconstruction $\wU\in \V_k^{\text{c}} (q)$ of the dG approximation
  $U\in \V_k^{\text{d}} (q-1).$ Then, the following maximal regularity
  a posteriori error estimate holds:
  \begin{equation}
    \label{aposteriori-estimate1}
    \|R\|_{L^p((0,t);X)}
    \leqslant  \|(u - \widehat U)'\|_{L^p((0,t);X)}
    +\|A(u - \widehat U)\|_{L^p((0,t);X)}
    \leqslant c_{p,X}\|R\|_{L^p((0,t);X)}
    % \quad\forall f\in L^p((0,\infty);X)
  \end{equation}
  for all $0<t\leqslant T$ for any $p\in (1,\infty)$ with the constant
  $c_{p,X}$ from \eqref{max-reg}.
  %$C$ depending only on $p$ and $X$.  
  Furthermore, the estimator is of optimal asymptotic order of accuracy,
  \begin{equation}
    \label{aposteriori-estimate2}
    \|R\|_{L^p((0,T);X)}
    \leqslant  C k^q\big (\|u^{(q+1)}\|_{L^p((0,T);X)}
    + \|Au^{(q)}\|_{L^p((0,T);X)} \big ), 
  \end{equation}
  provided that $u\in W^{q,p}\big ((0,T);\D(A)\big )\cap W^{q+1,p}\big ((0,T);X\big ).$ 
\end{theorem}

Our proofs of Theorems~\ref{Theorem1} and \ref{Theorem2} rely on the
maximal regularity property of the dG methods from \cite{AM-SINUM}
and on variational techniques. The variational technique is applicable
also in the case of nonlinear parabolic equations; see \cite{AL}; this
is our main motivation.

Similar error estimates were recently established in
\cite{AM-SINUM}. The proofs in \cite{AM-SINUM} rely on properties of
the Radau~IIA methods; in particular, the proof of the optimality of
$\|R\|_{L^p((0,T);X)}$ in \cite{AM-SINUM} is significantly lengthier
and more involved.

As the proof of Theorem~\ref{Theorem2} is very short, we give it here.

\begin{proof}
  Our assumption $u_0\in \D(A)$ ensures that the reconstruction
  $\widehat U$ of the dG approximation $U$ belongs to
  $\V_k^{\text{c}} (q).$ The residual $R,$ the amount by which
  $\widehat U$ misses being an exact solution of the differential
  equation in \eqref{ivp}, is a computable quantity, depending only on
  the numerical solution $\widehat U$ and the given forcing term $f.$
  Replacing $f$ in the residual by $u'+A u,$ we see that the error
  $\hat e:=u-\widehat U$ satisfies the \emph{error equation}
  \begin{equation}
    \label{eq:err-equation}
    \hat e'(t)+A\hat e(t)=-R(t), \quad t\in (t_n,t_{n+1}],
    \ n=0,\dotsc,N-1; \quad \hat e(0)=0.
  \end{equation}
  Now, the triangle inequality and the maximal $L^p$-regularity
  \eqref{max-reg} of the operator $A$ applied to the error equation
  \eqref{eq:err-equation} yield \eqref{aposteriori-estimate1}, i.e.,
  the asserted lower and upper a posteriori error estimators.
  
  In view of the representation \eqref{eq:err-equation} of $R,$ the
  optimality \eqref{aposteriori-estimate2} of the residual is an
  immediate consequence of the a priori error estimate
  \eqref{apriori-estimate2}.
\end{proof}

We present the a priori error analysis in Section~\ref{Se:2}. In
Section~\ref{Se:3} we extend these results to the case of
nonautonomous linear parabolic equations. The Appendix contains proofs
of relevant interpolation error estimates.

\section{A priori error estimates}\label{Se:2}

In this section we prove Theorem~\ref{Theorem1}. 

\subsection{A discrete $\ell^p(X)$-norm and its equivalence to the continuous $L^p(X)$-norm}\label{SSe:discr-norm}
We introduce the discrete $\ell^p(X)$-norm 
$\|\cdot \|_{\ell^p ((0,T);X )}$ on $\V_{k,0}^{\text{d}} (q-1):=\{v\in \V_k^{\text{d}} (q-1): v(0)=0\}$
and on $\V_{k,0}^{\text{c}} (q):=\{v\in \V_k^{\text{c}} (q): v(0)=0\}$ by
\begin{equation}
\label{eq:discrete-norm}
\|v \|_{\ell^p ((0,T);X )}:=\Big (\sum_{\ell=0}^{N-1} \Big (k\sum_{i=1}^q \|v(t_{\ell i})\|_{X}^p\Big )\Big)^{1/p}
%\quad \forall v\in \V_k^{\text{d}} (q-1);
\end{equation}
and show that it is equivalent to  the continuous $\|\cdot \|_{L^p ((0,T);X )}$-norm.

Let us focus on the space  $\V_{k,0}^{\text{d}} (q-1);$ the proof for the space $\V_{k,0}^{\text{c}} (q)$
is completely analogous.
First, it is easily seen that the continuous norm is dominated by the discrete norm,
\begin{equation}
\label{eq:discrete-norm-dom}
\|v \|_{L^p ((0,T);X)}\leqslant \tilde c_{q,p} \|v \|_{\ell^p ((0,T);X )}
\quad \forall v\in \V_{k,0}^{\text{d}} (q-1).
\end{equation}
Indeed, with $\ell_{mi}$ the Lagrange polynomials
$ \ell_i\in \P_{q-1}$ for the Radau points $c_1,\dotsc,c_q,$ shifted
to the subinterval $ J_m$, we have
\begin{equation*}
\begin{split}
 \int _{J_m}  \|  v(t) \|_{X}^p \, \d t   &=   \int _{J_m}  \Big\|
 \sum_{i=1}^q \ell_{mi}(t) v(t_{mi}) \Big\|_{X}^p \, \d t  
 \leqslant k \Big ( \sum_{i=1}^q  \| \ell_{mi}  \|_{L^\infty (J_m)} \|  v(t_{mi})\|_{X} \Big ) ^p\\
 %& \leqslant k \Big ( \sum_{i=1}^q  \| \ell_{mi}  \|_{L^\infty (J_m)} \|  v(t_{mi})\|_{X} \Big ) ^p\\
  & \leqslant  \Big ( \sum_{i=1}^q  \| \ell_{mi} \|^{p'} _{L^\infty (J_m)} \Big )^{p/p'}     
  \Big ( \sum_{i=1}^q  k  \|  v(t_{mi}) \|_{X}^p \Big )
   =c_{q,p}  k\sum_{i=1}^q  \|  v(t_{mi}) \|_{X}^p 
\end{split}
\end{equation*}
with $p'$ the dual exponent of $p, \frac 1p+ \frac 1{p'}=1,$ and summing over $m,$ we obtain \eqref{eq:discrete-norm-dom};
cf.~\cite{AM-SINUM}.

Next, we prove  that the discrete norm is dominated by the continuous norm
in the reference element $[0,1]$ for polynomials $v$ of degree at most $q-1$ 
with coefficients in $X;$ a scaling argument 
then shows that this is the case in arbitrary intervals $(0,T).$\footnote{We thank Pedro Morin 
for a personal communication regarding this result.}
We use the Lagrange form of $v,$
\[v(t)=\sum_{i=1}^q \ell_i(t) v(c_ i),\quad t\in [0,1].\]
%
%of polynomials of degree at most $q-1$ in $[0,1]$ with coefficients in $X.$
%Here, $c_i,i=1,\dotsc,q,$ are the Radau notes in $(0,1],$ and $\ell_i\in \P_{q-1}$
%are the corresponding Lagrange polynomials, $\ell_i(c_j)=\delta_{ij}.$

Let $i$ be such that
\begin{equation}
\label{eq:app-est1}
\| v(c_ i)\|_{X}=\max_{1\leqslant j\leqslant q}\| v(c_ j)\|_{X}.
\end{equation}
We want to show that
\begin{equation*}
\label{eq:app-est2}
\| v(c_ i)\|_{X}\leqslant c \| v\|_{L^p((0,1);X)}.
\end{equation*}
We have
\[\| v\|_{L^p((0,1);X)}^p\geqslant \int_{c_i-\delta}^{c_i} \| v(t)\|_{X}^p\, \d t
=\int_{c_i-\delta}^{c_i} \Big\|\ell_i(t)v(c_ i)+\sum_{\substack{j=1\\ j\ne i}}^q \ell_j(t)v(c_ j) \Big\|_{X}^p\, \d t,\]
whence, in view of \eqref{eq:app-est1},
\begin{equation}
\label{eq:app-est3}
\| v\|_{L^p((0,1);X)}^p\geqslant \| v(c_ i)\|_{X}^p
\int_{c_i-\delta}^{c_i} \Big (|\ell_i(t)|-\sum_{\substack{j=1\\ j\ne i}}^q |\ell_j(t)|\Big )^p\, \d t.
\end{equation}
Now, for any $\varepsilon_1, \varepsilon_2\in (0,1),$ for sufficiently small $\delta,$ we have
\begin{equation}
\label{eq:app-est4}
|\ell_i(t)|\geqslant 1-\varepsilon_1,\quad  |\ell_j(t)|\leqslant  \varepsilon_2,\quad j\ne i, \quad \forall t\in [c_i-\delta,c_i],
\end{equation}
and \eqref{eq:app-est3} yields
\begin{equation}
\label{eq:app-est5}
\| v\|_{L^p((0,1);X)}^p\geqslant 
\delta \big (1-\varepsilon_1-(q-1)\varepsilon_2\big )^p\| v(c_ i)\|_{X}^p
\end{equation}
and the desired property follows  easily.

The proof for the space $\V_{k,0}^{\text{c}} (q)$ is completely analogous.
In this case we shift the Lagrange polynomials $ \hat \ell_i\in \P_q$ for the points $c_0=0,c_1,\dotsc,c_q,$ 
to a subinterval  $ J_m$,  and use the fact that $t_{n,0}:=t_n=t_{n-1,q}.$

Notice that it is obvious from \eqref{eq:ext-interp} that the discrete $\ell^p(X)$-norms
of an element  $v\in \V_{k,0}^{\text{d}} (q-1)$ and of its reconstruction $\hat v\in \V_{k,0}^{\text{c}} (q)$
 coincide,
\begin{equation}
\label{eq:discrete-norm-equal}
\|\hat v \|_{\ell^p ((0,T);X )}=\|v \|_{\ell^p ((0,T);X )}
\quad \forall v\in \V_{k,0}^{\text{d}} (q-1).
\end{equation}

Since the equivalence constants of the discrete $\|\cdot \|_{\ell^p ((0,T);X )}$-  and
continuous $\|\cdot \|_{L^p ((0,T);X )}$-norms are independent of $T,$
the corresponding discrete $\|\cdot \|_{\ell^p ((0,t_n);X )}$-  and
continuous $\|\cdot \|_{L^p ((0,t_n);X )}$-seminorms, $n=1,\dotsc,N,$
are also equivalent with constants independent of $n.$
Of course, the  discrete $\|\cdot \|_{\ell^p ((0,t_n);X )}$-seminorm
is the term on the right-hand side of  \eqref{eq:discrete-norm}
with $N$ replaced by $n.$

\subsection{Maximal parabolic regularity of the dG method}\label{SSe:2.1}
We shall use the notation $\partial_k$ for the backward difference operator,
\[\partial_kv:=\frac {v (\cdot)- v (\cdot-k)}k,\quad v\in \V_{k,0}^{\text{d}} (q-1)\cup\V_{k,0}^{\text{c}} (q)\]
%
%for the  reconstruction $\wV$ of the dG approximation  $U\in  \V_k^{\text{d}} (q-1)$,  
with $v=0$ in the interval $[-k,0).$
Obviously, $\partial_k$ commutes with the reconstruction operator, $\partial_k \hat v=\widehat {\partial_k  v},
 v\in \V_{k,0}^{\text{d}} (q-1).$

In the case of vanishing initial value $u_0=0,$ for
the reconstruction $\wU$ and the dG approximation $U$ we have
the following maximal parabolic regularity result
%We recall the following maximal parabolic regularity result from
%\cite{AM-SINUM}: In the case of vanishing initial value $u_0=0,$ for
%the reconstruction $\wU$ and the dG approximation $U$ we have
%
\begin{equation}
\label{eq:max-reg-dG}
\begin{split}
\| \partial_k\wU\|_{\ell^p((0,T);X)}&+\| \wU '\|_{L^p((0,T);X)}+\|A \wU\|_{L^p((0,T);X)}\\ 
&+\|AU\|_{L^p((0,T);X)} \leqslant C_{p,X} \|f\|_{L^p((0,T);X)},
\end{split}
\end{equation}
with $C_{p,X}$ a method-dependent constant, independent of $N$ and $T.$  
The estimate for the last three terms on the left-hand side of \eqref{eq:max-reg-dG}
is given in \cite[(1.9) and the last line on p.\ 186]{AM-SINUM}; the proof relies on the interpretation 
of dG methods in \cite{AM-SINUM} as modified Radau~IIA methods and on the
maximal regularity property of Radau~IIA methods from \cite{KLL}.
Notice that $\widehat U\in \V_k^{\text{c}} (q)$ since $u_0=0\in \D(A).$ 

It remains to prove the estimate for the first term on the left-hand side of \eqref{eq:max-reg-dG}.
Now, the nodal values $\wU_{ni}=U_{ni}=U(t_{ni}), n=0,\dotsc,N-1, i=1,\dotsc,q,$ 
satisfy the Radau IIA equations with the nodal values $f(t_{ni})$ of the forcing term replaced 
 by the averages $f_{ni},$ 
\begin{equation}\label{average}
f_{ni}:=\frac 1 {\int _{J_n} \ell_{ni} (s) \, \d s }  \int _{J_n}  \ell_{ni} (s)f(s)  \, \d s,\quad
n=0,\dotsc,N-1,\ i=1,\dotsc,q;
%=\frac 1 {b_ik }  \int _{J_n}  \ell_{ni} (s) f(s)  \, \d s,\quad  i=1,\dotsc,q.
\end{equation}
see \cite[Lemma 2.2]{AM-SINUM}.   Therefore, according to the maximal regularity result
of  the Radau~IIA methods in \cite[Lemma 3.6]{KuLL}, we have
\begin{equation}
\label{eq:Rad-maxreg1}
\sum_{i=1}^q\|(\partial \wU_{ni})_{n=0}^{N-1}\|_{\ell^p(X)}^p\leqslant C_{p,r}\sum_{i=1}^q\|(  f_{ni})_{n=0}^{N-1}\|_{\ell^p(X)}^p;
\end{equation}
here, $\wU_{-1,1}=\dotsb=\wU_{-1,q}=0,$ and for a sequence $(v_n)_{n\in \N_0}\subset X,$ we used the notation 
\[\partial v_n:=\frac {v_n-v_{n-1}}k\quad\text{and}\quad \|(v_n)_{n=0}^{N-1}\|_{\ell^p(X)}:=\Big (k \sum_{n=0}^{N-1}\|v_n\|_X^p\Big )^{1/p}\]
for the backward difference quotient and for the discrete $\ell^p(X)$-norm $\|\cdot\|_{\ell^p(X)}.$

Notice that the term on the left-hand side of \eqref{eq:Rad-maxreg1} coincides with 
the first term on the left-hand side of \eqref{eq:max-reg-dG} raised to the power $p.$  Furthermore,
\begin{equation}
\label{bound_f_intro}
 \sum_{i=1}^q\|(  f_{ni})_{n=0}^{N-1}\|_{\ell^p(X)}^p  \leqslant  \gamma  \|f\|_{L^p((0,T);X)}^p %\quad N\in \N,
\end{equation}
with a constant $\gamma $ depending only on $c_1\dotsc,c_q,$ and $p;$
see \cite[(2.12)]{AM-SINUM}. Now, \eqref{eq:Rad-maxreg1} and \eqref{bound_f_intro}
yield
\[\| \partial_k\wU\|_{\ell^p((0,T);X)} \leqslant C_{p,X} \|f\|_{L^p((0,T);X)},\]
and the proof of \eqref{eq:max-reg-dG} is complete.

Obviously, \eqref{eq:max-reg-dG} is also valid with $T$ replaced by $t_n,n=1,\dotsc,N.$

Let us note that the estimate for the  first term on the left-hand side of \eqref{eq:max-reg-dG}
is not needed in the error analysis for linear parabolic equations; however, it plays a key
role  in the error analysis for nonlinear parabolic equations; see \cite{AL}.
The analogous estimate for Radau IIA methods from \cite[Lemma 3.6]{KuLL} 
was used there  in the error analysis for Radau IIA methods for nonlinear parabolic equations.

Notice also that due to the equivalence of the discrete $\ell^p(X)$- and the continuous $L^p(X)$-norms
established in section \ref{SSe:discr-norm}, $\| \partial_k\wU\|_{\ell^p((0,T);X)}$ can be replaced by
$\| \partial_k\wU\|_{L^p((0,T);X)}$ in the maximal regularity estimate \eqref{eq:max-reg-dG}.

Logarithmically quasi-maximal parabolic regularity results for dG
methods were earlier established in \cite{LV1} in general Banach
spaces for autonomous equations and in \cite{LV2} in Hilbert spaces
for nonautonomous equations. These works are not based on the
  maximal regularity theory and therefore cover also the cases of
variable time steps as well as the critical exponents $p=1,\infty.$

More recently, another approach to the maximal regularity of dG
  time discretization was presented in \cite{KK2024}. This is based on
  studying the dG approximation of a temporally regularized Green's
  function. The result allows quasi-uniform meshes but is restricted
  to $q\geqslant 2$.

\subsection{An interpolant and its approximation properties}\label{SSe:2.2}
We shall use a standard interpolant $\tilde u\in \V_k^{\text{d}} (q-1)$ of the solution $u$ such that
$\tilde u(t_n)=u(t_n), n=0,\dotsc,N,$ and $u-\tilde u$ is in each subinterval $J_n$ orthogonal
to polynomials of degree at most $q-2$ (with the second condition being void for $q=1$);
then, $\tilde u$ is determined in $J_n$ by the conditions
\begin{equation}
\label{eq:tilde-u1}
\left\{
\begin{aligned}
&\tilde u(t_{n+1})=u(t_{n+1}),\\
&\int_{J_n} \big (u(t)-\tilde u(t)\big )t^j\, \d t=0,\quad j=0,\dotsc,q-2;
\end{aligned}
\right.
\end{equation}
cf., e.g., \cite[(12.9)]{T} and \cite[\textsection\textsection 3.1--3.3]{SS} for the case of Hilbert spaces.

The approximation property (valid for $1\leqslant p\leqslant\infty$,
$q\geqslant 1$)
\begin{equation}
\label{eq:estim-desired1}
\|u-\tilde u\|_{L^p((0,T);X)} \leqslant Ck^{q} \|u^{(q)}\|_{L^p((0,T);X)} 
\end{equation}
will play an important role in our analysis.
Furthermore, we will use the following approximation properties of the
reconstruction $\hat {\tilde u}$ of $\tilde u$, 
\begin{align}
\label{eq:approx-prop-desired1}
\| u-\hat {\tilde u}\|_{L^p((0,T);X)}
&
\leqslant C k^{q}\|u^{(q)}\|_{L^p((0,T);X)},
\\
\label{eq:approx-prop-desired2}
\| (u-\hat {\tilde u})'\|_{L^p((0,T);X)}
&
\leqslant C k^{q}\|u^{(q+1)}\|_{L^p((0,T);X)}.  
\end{align}
We shall prove existence and uniqueness of $\tilde u$ as well as the
approximation properties \eqref{eq:estim-desired1},
\eqref{eq:approx-prop-desired1}, and \eqref{eq:approx-prop-desired2}
in the appendix.

Notice that the orthogonality condition in \eqref{eq:tilde-u1} can be
equivalently written in the form
\begin{equation}
\label{eq:tilde-u3}
\int_{J_n}   \<u(t)-\tilde u(t) ,v(t) \>  \, \d t
=0 \quad \forall v \in \P_{X'}(q-2).
\end{equation}

\subsection{A priori error estimates}\label{SSe:2.3}
Using the interpolant $\tilde u,$ we decompose the error $e=u-U$ in the form
\begin{equation*}
e=\rho + \vartheta
\quad\text{with}
\quad \rho:=u-\tilde u
\quad\text{and}
\quad
\vartheta:=\tilde u-U\in \V_k^{\text{d}} (q-1).
\end{equation*}
The desired estimate for the interpolation error $\rho,$ 
\begin{equation}
\label{eq:rho-estimate}
\|A\rho\|_{L^p((0,T);X)}\leqslant C k^q\|Au^{(q)}\|_{L^p((0,T);X)}, 
\end{equation}
is obtained by replacing $u$ by $Au$ in \eqref{eq:estim-desired1}.
Therefore, to prove  \eqref{apriori-estimate1}, it remains to bound $\vartheta.$ 

Since $\rho=u-\tilde u$ is orthogonal to $v'$ in $J_n$ for any
$v\in \P_{X'}(q-1)$, cf.~\eqref{eq:tilde-u3}, and
vanishes at the nodes $t_n,$ integration by parts shows that it has
the following crucial property,
\begin{equation}
\label{eq:rho1}
\int_{J_n}   \< \rho' ,v \>  \, \d t
+ \< \rho_n^{+}-\rho_n, v_n^{+}\>
=0 \quad \forall v \in \P_{X'}(q-1);
\end{equation}
cf.\ \cite[(12.13)]{T}.

Subtracting the dG method \eqref{dg} from the corresponding relation
for the exact solution, and using the splitting $e=\rho+\vartheta$ as
well as relation \eqref{eq:rho1} for $\rho,$ we obtain the following
equation for $\vartheta,$
\begin{equation}
\label{eq:theta1}
\int_{J_n}   \big( \< \vartheta' ,v \>  + \< A\vartheta ,v \> \big) \, \d t  
+ \< \vartheta_n^{+}-\vartheta_n, v_n^{+}\>
= \int_{J_n}  \< A\rho ,v \> \, \d t \quad \forall v \in \P_{X'}(q-1)
\end{equation}
for $n=0,\dotsc,N-1$.  Notice that the reconstruction
$\hat \vartheta=\hat{\tilde u}-\wU$ of $\vartheta$ belongs to
$ \V_k^{\text{c}} (q)$ since $\vartheta(0)=0\in \D(A);$ in contrast,
$\hat{\tilde u}$ and $\wU$ belong only to $\X_k^{\text{c}} (q)$ for
$u_0\in X.$

Since $\vartheta(0)=0,$ the maximal regularity property
\eqref{eq:max-reg-dG} of the dG method applied to the error equation
\eqref{eq:theta1} yields
\begin{equation*}
\|\hat\vartheta'\|_{L^p((0,T);X)}
  +\|A \hat\vartheta\|_{L^p((0,T);X)}
  +\|A \vartheta\|_{L^p((0,T);X)}
  \leqslant C_{p,X} \|A\rho\|_{L^p((0,T);X)},
\end{equation*}
and thus, in view of \eqref{eq:rho-estimate}, 
\begin{equation}
\label{eq:theta2}
\|  \hat\vartheta'\|_{L^p((0,T);X)}
+\|A \hat\vartheta\|_{L^p((0,T);X)}
+\|A \vartheta\|_{L^p((0,T);X)}
\leqslant C k^q\|Au^{(q)}\|_{L^p((0,T);X)}.
\end{equation}
Now, \eqref{apriori-estimate1} follows immediately from
\eqref{eq:rho-estimate} and \eqref{eq:theta2}.  To prove
\eqref{apriori-estimate2}, we write
\begin{equation*}
  \hat{e}
  =u-\widehat{U}
  =(u-\hat{\tilde{u}})
  +(\hat{\tilde{u}}-\widehat{U})
  =\hat{\rho}+\hat{\vartheta}
\end{equation*}
and combine \eqref{eq:theta2} with the estimates for
$\hat{\rho}:=u-\hat{\tilde u}$ and $\hat{\rho}'$ from
\eqref{eq:approx-prop-desired1} and \eqref{eq:approx-prop-desired2},
namely
\begin{align*}
\|A\hat{\rho}\|_{L^p((0,T);X)}
\leqslant C k^q\|Au^{(q)}\|_{L^p((0,T);X)}, \quad
\|\hat{\rho}'\|_{L^p((0,T);X)}
\leqslant C k^q\|u^{(q+1)}\|_{L^p((0,T);X)}. 
\end{align*}

\section{Extension to nonautonomous equations}\label{Se:3}

In this section, we extend the maximal parabolic regularity stability
estimates for dG methods to nonautonomous parabolic equations by a
perturbation argument. For similar ideas and results, we refer to
\cite{LV2} and \cite[\textsection 3.6]{KuLL}, \cite{AM-NM} for the dG
method with piecewise constant elements and for Radau~IIA methods,
respectively.  Furthermore, we establish optimal order a priori and a
posteriori error estimates.

We consider an initial value problem for a nonautonomous linear parabolic equation,
\begin{equation}
\label{ivp-na}
\left \{
\begin{aligned} 
&u' (t) +A(t) u(t) =f(t), \quad 0<t< T,\\
&u(0)=u_0,
\end{aligned}
\right .
\end{equation}
in a UMD Banach space $X.$ 

Our structural assumptions on $A(t)$ are that all operators
$A(t), t\in [0,T],$ share the same domain $\D(A),$ $A(t)$ is the
generator of an analytic semigroup on $X$ having maximal
$L^p$-regularity, for every $t\in [0,T],$ $A(t)$ induce equivalent
norms on $\D(A),$
\begin{equation}
\label{norms-na}
\|A(t)v\|_X\leqslant c \|A(\tilde t)v\|_X
\quad \forall t, \tilde t\in [0,T] \  \forall v\in \D(A),
\end{equation}
and $A(t)\colon \D(A)\to X$ satisfies the Lipschitz condition with
respect to $t,$
\begin{equation}
\label{Li-na}
\|\big (A(t)-A(\tilde t)\big ) v\|_X
\leqslant L|t-\tilde t| \|A(s)v\|_X
\quad \forall  t,  \tilde t \in [0, T] \  \forall v\in \D(A),
\end{equation}
for all $s\in [0,T].$

With starting value $U(0)=U_0:=u_0\in X$ and source term
$f\in L^p((0,T);X)$, we consider the discretization of the initial
value problem \eqref{ivp-na} by the dG method dG$(q-1)$, i.e., we seek
$U\in \V_k^{\text{d}} (q-1)$ such that
\begin{equation}
\label{eq:dg-na1}
\int_{J_n}   \big( \< U' ,v \>  + \< A(t)U ,v \> \big) \, \d t  
+ \< U_n^{+}-U_n, v_n^{+}\>
= \int_{J_n} \<f,v\>\, \d t \quad \forall v \in \P_{X'}(q-1)
\end{equation}
for $n=0,\dotsc,N-1$; cf.\ \eqref{dg}.

\subsection{Maximal parabolic regularity}\label{SSe:4.1}
Here, for vanishing starting value $U_0,$ we establish maximal parabolic regularity 
of dG methods for the nonautonomous parabolic equation \eqref{ivp-na} via 
a perturbation argument. 

For a fixed $t_m,m\in \{1,\dotsc,N\},$ we write \eqref{eq:dg-na1} in the form
\begin{equation}
\label{eq:dg-na2}
\begin{aligned}
 \int_{J_n}    \< U' ,v \>\, \d t &+  \int_{J_n}  \< A(t_m)U ,v   \>\, \d t 
  + \< U_n^{+}-U_n, v_n^{+}\>\\
  & =\int_{J_n}  \< (A(t_m)-A(t))U ,v   \>\, \d t 
 + \int_{J_n}  \< f  ,v   \>\, \d t\quad \forall v \in \P_{X'}(q-1),
\end{aligned}
\end{equation}
$n=0,\dotsc, m-1.$ Since the time $t$ is frozen at $t_m$ on the
left-hand side of \eqref{eq:dg-na2}, we can apply the maximal
parabolic regularity estimate \eqref{eq:max-reg-dG} for dG methods for
autonomous equations and obtain
\begin{equation}
\label{eq:dg-na3}
\begin{split}
\| \partial_k\wU\|_{\ell^p((0,t_\ell);X)}&+\| \wU '\|_{L^p((0,t_\ell);X)}
+\|A(t_m) \wU\|_{L^p((0,t_\ell);X)}+E_\ell\\
&\leqslant C_{p,X}\big (G_\ell+ \|f\|_{L^p((0,t_\ell);X)}\big ), \quad \ell=1,\dotsc,m,  
\end{split}
\end{equation}
with
\begin{equation*}
G_\ell:=\| (A(t_m)-A(\cdot))U\|_{L^p((0,t_\ell);X)},
\ E_\ell:=\|A(t_m) U\|_{L^p((0,t_\ell);X)},\  \ell=1,\dotsc,m, 
\end{equation*}
and $E_0:=0.$

Now, according to \eqref{eq:dg-na3},
\begin{equation}
\label{eq:dg-na5}
E_m^p \leqslant C\big (G_m^p+ \|f\|_{L^p((0,t_m);X)}^p\big ).
\end{equation}
Furthermore,
\begin{equation*}
  G_m^p=\| (A(t_m)-A(\cdot))U\|_{L^p((0,t_m);X)}^p
  =\sum_{\ell=0}^{m-1}\| (A(t_m)-A(\cdot))U\|_{L^p(J_\ell;X)}^p,
\end{equation*}
and thus, in view of the Lipschitz condition  \eqref{Li-na},
\begin{equation*}
  G_m^p
  \leqslant L^p \sum_{\ell=0}^{m-1}(t_m-t_\ell)^p \|
  A(t_m)U\|_{L^p(J_\ell;X)}^p
  =L^p \sum_{\ell=0}^{m-1}(t_m-t_\ell)^p (E_{\ell+1}^p-E_\ell^p).  
\end{equation*}
Hence, by summation by parts, we have
\begin{equation}\label{eq:dg-na6} 
  G_m^p \leqslant L^p\sum_{\ell=1}^{m} a_\ell E_\ell^p
  =L^p\sum_{\ell=1}^{m-1} a_\ell E_\ell^p+L^pk^pE_m^p,  
\end{equation}
with $a_\ell :=(t_m-t_{\ell-1})^p-(t_m-t_\ell)^p,$ and \eqref{eq:dg-na5} yields
\begin{equation}\label{eq:dg-na7}
  E_m^p
 \leqslant
  C\|f\|_{L^p((0,t_m);X)}^p 
  +C \sum_{\ell=1}^{m-1} a_\ell E_\ell^p 
  +Ck^pE_m^p 
\end{equation}
and hence, for $Ck^p\leqslant 1/2$,
\begin{equation*}
  E_m^p\leqslant C\|f\|_{L^p((0,t_m);X)}^p
  +C \sum_{\ell=1}^{m-1} a_\ell E_\ell^p. 
\end{equation*}
Since the sum $\sum_{\ell=1}^m a_\ell$ is uniformly bounded,
\begin{equation*}
\sum_{\ell=1}^m a_\ell=\big (t_m-t_0)^p\leqslant T^p,  
\end{equation*}
a discrete Gronwall-type argument applied to \eqref{eq:dg-na7} leads to
\begin{equation}\label{eq:dg-na8}
E_m^p\leqslant C\|f\|_{L^p((0,t_m);X)}^p.
\end{equation}

Combining \eqref{eq:dg-na8} with \eqref{eq:dg-na6} and
\eqref{eq:dg-na3}, we obtain, for sufficiently small $k$, the desired
maximal parabolic regularity stability estimate
\begin{equation}
\label{eq:dg-na9}
\begin{split}
\| \partial_k\wU\|_{\ell^p((0,t_m);X)}&+\| \wU '\|_{L^p((0,t_m);X)}+\|A(t_m) \wU\|_{L^p((0,t_m);X)} +\|A(t_m)U\|_{L^p((0,t_m);X)}\\
&\leqslant C_{p,X,T} \|f\|_{L^p((0,t_m);X)}, \quad m=1,\dotsc,N, 
\end{split}
\end{equation}
%
%$m=1,\dotsc,N,$ 
with a constant $C_{p,X,T}$ independent of $k.$ Notice
that, due to the equivalence of norms \eqref{norms-na}, $A(t_m) $ can
be replaced by $A(s)$ on the left-hand side of \eqref{eq:dg-na9}, for
arbitrary $s\in [0,T].$

\subsection{A priori error estimates}\label{SSe:4.2}
Here, we establish the analogue of Theorem~\ref{Theorem1} in the
nonautonomous case.

\begin{theorem}[A priori error estimates]\label{Theorem3} 
  Assume that the solution of \eqref{ivp-na} is sufficiently regular,
  $u\in W^{q,p}\big ((0,T);\D(A)\big ).$ Then, for sufficiently small
  $k$, the dG approximation $U\in \V_k^{\text{d}} (q-1)$ of
  \eqref{eq:dg-na1} satisfies the estimate
  \begin{equation}
    \label{apriori-estimate1-na}
    \|A(s)(u-U)\|_{L^p((0,T);X)}  \leqslant C k^q\|A(s)u^{(q)}\|_{L^p((0,T);X)}
  \end{equation}
  for any $s\in [0,T].$ Furthermore, if $u_0\in \D(A)$ and
  $u\in W^{q+1,p}\big ((0,T);X\big ),$ for the reconstruction
  $\wU\in \V_k^{\text{c}} (q)$ of $U,$ we have
  \begin{equation}
    \label{apriori-estimate2-na}
    \begin{aligned}
      & \|(u - \widehat U)'\|_{L^p((0,T);X)} + \|A(s)(u-\wU)\|_{L^p((0,T);X)}
      \\ & \qquad
      \leqslant
      C k^q\big( \|u^{(q+1)}\|_{L^p((0,T);X)}+ \|A(s)u^{(q)}\|_{L^p((0,T);X)}\big).
    \end{aligned}
  \end{equation}
  The constant $C$ depends on $q$, $p$, $L$, $X$, and $T$, but it is
  independent of the solution $u$ and of the time step $k.$
\end{theorem}

\begin{proof}
  We proceed as in the proof of Theorem~\ref{Theorem1}. In particular,
  we decompose the error $e=u-U$ in the form
  \begin{equation*}
    e=\rho+\vartheta\quad\text{with} \quad 
    \rho:=u-\tilde u 
    \quad\text{and}\quad
    \vartheta:=\tilde u-U\in \V_k^{\text{d}} (q-1) . 
  \end{equation*}
  The error $\rho$ can be estimated as in Section~\ref{SSe:2.3}; for
  instance, the analogue of \eqref{eq:rho-estimate} reads
  \begin{equation}
    \label{eq:rho-estimate-na}
    \|A(s)(u - \tilde u)\|_{L^p((0,T);X)}\leqslant c k^q\|A(s)u^{(q)}\|_{L^p((0,T);X)},
  \end{equation}
  for any $s\in [0,T].$

  Furthermore, subtracting the dG method \eqref{eq:dg-na1} from the
  corresponding equation for the exact solution
  $u$ of \eqref{ivp-na} and using the identity \eqref{eq:rho1} for
  $\rho,$ we obtain the following equation for $\vartheta,$
  \begin{equation}
    \label{eq:theta1-na}
    \int_{J_n}   \big( \< \vartheta' ,v \>  + \< A(t)\vartheta ,v \> \big) \, \d t  
    + \< \vartheta_n^{+}-\vartheta_n, v_n^{+}\>
    = \int_{J_n}  \< A(t)\rho ,v \> \, \d t \quad \forall v \in \P_{X'}(q-1)
  \end{equation}
  for $n=0,\dotsc,N-1$.  
  
  Since $\vartheta(0)=0,$ the maximal regularity property of the dG
  method for nonautonomous equations (see \eqref{eq:dg-na9}) applied
  to the error equation \eqref{eq:theta1-na} yields
  \begin{equation*}
    \|\hat\vartheta'\|_{L^p((0,T);X)}
    +\|A(s)\hat\vartheta\|_{L^p((0,T);X)}
    +\|A(s) \vartheta\|_{L^p((0,T);X)}
    \leqslant C_{p,X} \|A(\cdot)\rho\|_{L^p((0,T);X)},  
  \end{equation*}
  for any $s\in [0,T],$ and thus, in view of \eqref{eq:rho-estimate-na}, 
  \begin{equation}
    \label{eq:theta2-na}
    \|\hat\vartheta'\|_{L^p((0,T);X)}
    +\|A(s) \hat\vartheta\|_{L^p((0,T);X)}
    +\|A(s) \vartheta\|_{L^p((0,T);X)}
    \leqslant C k^q\|A(s) u^{(q)}\|_{L^p((0,T);X)}.
  \end{equation}
  The proof can now be completed as in the case of Theorem~\ref{Theorem1}.
\end{proof}

\subsection{A posteriori error estimates}\label{SSe:4.3}
Let $R$ be the \emph{residual} of the reconstruction $\widehat U,$
\begin{equation*}
R(t):= \widehat U'(t)+A(t)\widehat U(t)-f(t), \quad t\in (t_n,t_{n+1}], \quad n=0,\dotsc,N-1.
\end{equation*}
Then, the error $\hat e:=u-\widehat U$ satisfies the \emph{error equation}
\begin{equation}
\label{eq:err-eq-1-na}
\hat e'(t)+A(t)\hat e(t)=-R(t), \quad t\in (t_n,t_{n+1}], \ n=0,\dotsc,N-1; \quad \hat e(0)=0.
\end{equation}
Let us now fix an $s\in (0,T].$ To apply the \emph{maximal
  $L^p$-regularity} of the operator $A(s),$ frozen at time $s,$ we
rewrite \eqref{eq:err-eq-1-na} in the form
\begin{equation*}
\hat e'(t)+A(s)\hat e(t)=[A(s)-A(t)]\hat e(t)-R(t), \quad t\in (0,s].
\end{equation*}
Proceeding as in the proof of the maximal regularity property for
nonautonomous parabolic equations in the continuous case, cf., e.g.,
\cite[\textsection 4.2]{AM-NM} for the derivation a posteriori
estimates for Radau~IIA methods, we obtain the desired a posteriori
error estimate
\begin{equation*}
\|\hat e'\|_{L^p((0,s);X)}+\|A(s)\hat e\|_{L^p((0,s);X)}\leqslant C\|R\|_{L^p((0,s);X)},
\quad 0<s\leqslant T,
\end{equation*}
for any $p\in (1,\infty),$ with a constant $C$ depending on $p, X, L,$
and $T,$ but independent of $s.$

Furthermore, the triangle inequality applied to \eqref{eq:err-eq-1-na}
yields a lower a posteriori error estimator,
\begin{equation*}
\|R\|_{L^p((0,s);X)} \leqslant  \|\hat e'\|_{L^p((0,s);X)}+\|A(\cdot)\hat e\|_{L^p((0,s);X)},
\quad 0<s\leqslant T.
\end{equation*}

As in the autonomous case, we see that the a posteriori error
estimator is of optimal order as an immediate consequence of the a
priori error estimate \eqref{apriori-estimate2-na}.

\appendix

\section{Interpolation error estimates}
\label{Se:5}

We prove error estimates for $\tilde{v}$ and $\hat{\tilde{v}}$.   We
shall use a standard argument based on the Bramble--Hilbert lemma;
cf.~\cite[(4.4.4)]{BrennerScott}. The key ingredients are reproduction
of polynomials and boundedness with respect to a relevant Sobolev
norm.  

\subsection{Existence and uniqueness of \texorpdfstring{$\bm{\tilde{v}}$}{\tilde v}}
\label{SSe:A.1}

The interpolant $\tilde v$ of a function $v \in C([0,T];X)$ can be
expressed in terms of the value $v(t_{n+1})$ of $v$ at $t_{n+1}$ and
the Legendre coefficients $v_0,\dotsc,v_{q-2}\in X$ of $v,$
\begin{equation}
\label{eq:Legendre-coeff} 
v_i:=\frac 1{\|L_{ni}\|_{L^2(J_n)}^2}\int_{J_n} L_{ni}(t)v(t)\, \d t.
\end{equation}
More precisely, for $t\in J_n$, 
\begin{equation}
\label{eq:tilde-u2}
\tilde v(t)=(P_{q-2}v)(t)+L_{n,q-1}(t)\Big [v(t_{n+1})-\sum_{i=0}^{q-2}v_i\Big ],\quad
P_{q-2}v=\sum_{i=0}^{q-2}L_{ni} v_i,
\end{equation}
with $P_{q-2}$ the piecewise $L^2$-projection onto $\X_k^{\text{d}}(q-1)$;
here, $L_{ni}$ are the Legendre polynomials $L_i$ of degree $i$ shifted to the interval $J_n,$
\begin{equation*}
L_{ni}\big (\tfrac 12(t_n+t_{n+1}+ks)\big )=L_i(s), \quad s\in [-1,1];
\end{equation*}
see \cite[(3.2)]{SS}.
Indeed, the part $P_{q-2}v$ of $\tilde v$ on the right-hand side of \eqref{eq:tilde-u2}
is due to the orthogonality of $v-\tilde v$ to the orthogonal polynomials $L_{ni}, i=0,\dotsc,q-2.$
The coefficient of $L_{n,q-1}$ is determined by the interpolation condition 
$\tilde v(t_{n+1})=v(t_{n+1})$ and the property $L_{ni}(t_{n+1})=1$
of the Legendre polynomials, which yields $(P_{q-2}v)(t_{n+1})=v_0+\dotsb+v_{q-2}.$
In particular,  $\tilde v=v$ for  $ v\in \X_k^{\text{d}} (q-1).$

\subsection{Proof of the approximation property \eqref{eq:estim-desired1}}\label{SSe:A.2}

We first consider interpolation of functions on the unit interval
$(0,1)$.  Here we assume $1\leqslant p\leqslant\infty$ and $q\geqslant 1$
and hence Sobolev's embedding ${W^{q,p}((0,1);X)}\subset C([0,1];X)$
holds. From \eqref{eq:Legendre-coeff} and \eqref{eq:tilde-u2} it is
clear that the interpolation operator
$C([0,1];X)\to {W^{q,p}((0,1);X)}, v\mapsto \tilde{v},$ is bounded, so
that by Sobolev's inequality,
\begin{equation}
  \label{eq:2}
  \| \tilde{v} \|_{W^{q,p}((0,1);X)}
  \leqslant C \| v \|_{C([0,1];X)}
  \leqslant C \| v \|_{W^{q,p}((0,1);X)} .
\end{equation}
Moreover, 
\begin{equation}
  \label{eq:2b}
  \tilde{v}=v \quad \forall v\in\P_X(q-1). 
\end{equation}
Hence, by a standard argument based on the Bramble--Hilbert lemma, we
have
\begin{equation}
  \label{eq:3}
    \| v-\tilde{v} \|_{W^{q,p}((0,1);X)}
    \leqslant C | v |_{W^{q,p}((0,1);X)}.  
\end{equation}
Here $| v |_{W^{q,p}((0,1);X)}=\| v^{(q)} \|_{L^{p}((0,1);X)}$ denotes the seminorm.
  
In fact, by the Bramble--Hilbert lemma there is a Taylor polynomial
$\bar{v}\in\P_X(q-1)$ such that
\begin{equation}
  \label{eq:6}
  \| v-\bar{v} \|_{W^{q,p}((0,1);X)} \leqslant C | v |_{W^{q,p}((0,1);X)}.  
\end{equation}
Hence, since by \eqref{eq:2b} $\tilde{\bar{v}}=\bar{v}$, and by
\eqref{eq:2}, \eqref{eq:6},
\begin{align*}
  \| v-\tilde{v} \|_{W^{q,p}((0,1);X)}
  &
    \leqslant 
  \| v-\bar{v} \|_{W^{q,p}((0,1);X)}
  +
  \| \bar{v}-\tilde{v} \|_{W^{q,p}((0,1);X)}
  \\
  &
  =
  \| v-\bar{v} \|_{W^{q,p}((0,1);X)}
  +
  \| (\bar{v}-{v})^{\tilde{}} \,\|_{W^{q,p}((0,1);X)}
  \\
  &
    \leqslant 
  C\| v-\bar{v} \|_{W^{q,p}((0,1);X)}
  \leqslant C | v |_{W^{q,p}((0,1);X)},  
\end{align*}
which is \eqref{eq:3}.

Since $| v |_{W^{j,p}((0,1);X)}\leqslant \| v \|_{W^{q,p}((0,1);X)}$
and $| v |_{L^\infty((0,1);X)}\leqslant C\| v \|_{W^{q,p}((0,1);X)}$,
from \eqref{eq:3} we now conclude
\begin{align*}
    | v-\tilde{v} |_{W^{j,p}((0,1);X)}
    &
    \leqslant C | v |_{W^{q,p}((0,1);X)}, \quad j=0,\dotsc,q, 
    \\
    \| v-\tilde{v} \|_{L^\infty((0,1);X)}
    &
    \leqslant C | v |_{W^{q,p}((0,1);X)}.  
\end{align*}
Finally, by a change of variables and with a slight abuse of notation,
we have 
\begin{align*}
  | v |_{W^{j,p}((0,1);X)} & = k^{j-1/p}| v |_{W^{j,p}(J_n;X)}, \\
  \| v \|_{L^\infty((0,1);X)} & = \| v-\tilde{v} \|_{L^\infty(J_n;X)} .  
\end{align*}
This proves 
\begin{align}
  \label{eq:5a}
  | v-\tilde{v} |_{W^{j,p}(J_n;X)}
  & \leqslant C k^{q-j}| v |_{W^{q,p}(J_n;X)},
    \quad j=0,\dotsc,q, \\
  \label{eq:5b}
  \| v-\tilde{v} \|_{L^\infty(J_n;X)}
  & \leqslant C k^{q-1/p}| v |_{W^{q,p}(J_n;X)} . 
\end{align}
As a consequence of \eqref{eq:5a} with $j=0$ (or of \eqref{eq:5b}) we have
\eqref{eq:estim-desired1}.  This is the only case that we need in the
present work; the other cases are included for the sake of
completeness.  

The approximation property \eqref{eq:5b} for general $p$ and for
Banach spaces is an analogue of \cite[(12.10)]{T} for $p=2$ and for
Hilbert spaces.

\subsection{Proof of the approximation properties
  \eqref{eq:approx-prop-desired1} and \eqref{eq:approx-prop-desired2}}
\label{SSe:A.3}

We begin by proving that the operator $v\mapsto\hat{\tilde{v}}$
reproduces polynomials of degree at most $q$. Again we assume
$1\leqslant p\leqslant\infty$.

\begin{lemma}[Reproduction property]\label{Le:reprod} 
Assume that $v\in \X_k^{\text{c}} (q).$ Then, the following
reproduction property holds:
\begin{equation}
\label{eq:reprod1}
\hat{\tilde{v}}=v.
\end{equation}
In particular, if $v\in \X_k^{\text{c}} (q-1)$, then
$\hat{\tilde{v}}=\tilde{v}=v$.
\end{lemma}

\begin{proof}
  Let us consider the Lagrange interpolant
  $w\in\X_k^{\text{d}} (q-1),$ with $w(0)=v(0),$ of $v$ at the Radau
  nodes,
  \begin{equation}\label{new-interpolant}
    w ( t_{nj})= v( t_{nj}), \quad j=1, \dotsc ,q,
  \end{equation}
  for $n=0,\dotsc,N-1.$ Then, it follows immediately from the definition
  \eqref{eq:ext-interp} that $v\in \X_k^{\text{c}} (q)$ is the
  reconstruction of $w\in\X_k^{\text{d}} (q-1)$, that is, $v=\widehat w$.
  
  It remains to show that $\tilde v=w,$ that is, that $w$ satisfies the
  orthogonality condition
  \begin{equation*}
    \int_{t_n}^{t_{n+1}} \big (v(t)-w(t)\big ) t^j\, \d t=0,\quad j=0,\dotsc,q-2.
  \end{equation*}
  This can be easily seen by using the Lagrange form of the
  interpolation error $v(t)-w(t)=(v_n-w_n^+)\ell_{n0}(t),$ with
  $ \ell_{n0}$ the polynomial of degree $q$ vanishing at the nodes
  $t_{n1},\dotsc,t_{nq}$ and taking the value $1$ at the node
  $t_{n0},$ and the orthogonality of $\ell_{n0}$ to polynomials of
  degree at most $q-2$ or, equivalently, by the exactness of the Radau
  quadrature rule for polynomials of degree $2q-2$.
  
  For $v\in \X_k^{\text{c}} (q-1)$, we obviously have $w=v$ by
  \eqref{new-interpolant}, i.e., $\tilde v=v.$
\end{proof}

We next consider the boundedness of $v\mapsto\hat{\tilde{v}}$.  Let
$v\in C([0,T];X)$. The hat operator interpolates $\tilde{v}$ at the
node values $\hat{\tilde{v}}(t_{nj})=\tilde{v}(t_{nj})$,
$j=0,\dotsc,q$, where $\tilde{v}(t_{n0})=\tilde{v}(t_{n})=v(t_n)$ and
$\tilde{v}(t_{nq})=\tilde{v}(t_{n+1})=v(t_{n+1})$. Thus, in view of
\eqref{eq:Legendre-coeff} and \eqref{eq:tilde-u2},
$\hat{\tilde{v}}|_{J_n}$ depends only on the values of $v$ in
$[t_n,t_{n+1}]$.  More precisely, after a transformation to the unit
interval, we have
\begin{equation*}
  \| \hat{\tilde{v}} \|_{W^{q+1,p}((0,1);X)}
  \leqslant C \| v \|_{C([0,1];X)}
  \leqslant C \| v \|_{W^{q+1,p}((0,1);X)} .
\end{equation*}
Moreover, according to \eqref{eq:reprod1},
\begin{equation*}
  \hat{\tilde{v}}=v \quad \forall v\in\P_X(q). 
\end{equation*}
By the same argument as in Subsection~\ref{SSe:A.2}, we conclude
\begin{align*}
    | v-\hat{\tilde{v}} |_{W^{j,p}((0,1);X)}
  & \leqslant C | v |_{W^{q+1,p}((0,1);X)}, \quad j=0,\dotsc,q+1, \\ 
  \| v-\hat{\tilde{v}} \|_{L^\infty((0,1);X)}
  & \leqslant C | v |_{W^{q+1,p}((0,1);X)} ,
\end{align*}
which proves
\begin{align*}
  | v-\hat{\tilde{v}} |_{W^{j,p}(J_n;X)}
  & \leqslant C k^{q+1-j}| v |_{W^{q+1,p}(J_n;X)}, \quad
  j=0,\dotsc,q+1, \\
  \| v-\hat{\tilde{v}} \|_{L^\infty(J_n;X)}
  & \leqslant C k^{q+1-1/p}| v |_{W^{q+1,p}(J_n;X)} . 
\end{align*}
In particular, for $j=1$, we obtain \eqref{eq:approx-prop-desired2}. 

Since $v\mapsto\hat{\tilde{v}}$ also reproduces polynomials of degree
at most $q-1$, the same argument shows
\begin{align}
  \label{eq:5ae}
  | v-\hat{\tilde{v}} |_{W^{j,p}(J_n;X)}
  & \leqslant C k^{q-j}| v |_{W^{q,p}(J_n;X)}, \quad j=0,\dotsc,q, \\
    \label{eq:5af} 
    \| v-\hat{\tilde{v}} \|_{L^\infty(J_n;X)}
  & \leqslant C k^{q-1/p}| v |_{W^{q,p}(J_n;X)} . 
\end{align}
From \eqref{eq:5ae} for $j=0$ (or from \eqref{eq:5af}), we obtain
\eqref{eq:approx-prop-desired1}.
 
%\newpage
\bibliographystyle{amsplain}

\end{document}